\documentclass[10pt,oneside]{amsart}

\usepackage{amssymb,amsmath,amsthm}
\usepackage[english]{babel}
\usepackage[dvips]{graphicx}
\usepackage{graphicx}
\usepackage{pb-diagram}
\usepackage{picinpar}
\usepackage[latin1]{inputenc}
\usepackage{layout}

\newtheorem{theo}{Theorem}
\newtheorem{lem}{Lemma}
\newtheorem{prop}{Proposition}

\numberwithin{equation}{section}
\numberwithin{theo}{section}

\def\home {\hbox{ Homeo }}

\def\ord{\hbox{ord}}

\def\sign { \hbox{sign}}

\DeclareMathOperator{\auto}{Aut}

\DeclareMathOperator{\isom}{Isom}

\title{Equivalence of Group Actions on  Riemann Surfaces}
\author{Mariela Carvacho Bustamante}

\address{Universidad de Chile}
\email{marie.carvacho@gmail.com}
\date{\today}

\thanks{Special thanks to my advisor Rub\'i Rodr\'iguez. }

\begin{document}

\begin{abstract}
\noindent We produce for each natural number $n \geq 3$ two 1--parameter families of Riemann surfaces admitting automorphism groups with two cyclic subgroups $H_{1}$ and $H_{2}$ of orden $2^{n}$, that are conjugate in the group of orientation--preserving homeomorphism of the corresponding Riemann surfaces, but not conjugate in the group of conformal automorphisms.

\noindent This property  implies that the subvariety $\mathcal{M}_{g}(H_{1})$ of the moduli space $\mathcal{M}_{g}$ consisting of the points representing the Riemann surfaces of genus $g$ admitting a group of automorphisms topologically conjugate to $H_{1} $ (equivalently to $H_{2}$ ) is not a normal subvariety.
\end{abstract}

\maketitle

\section{Introduction}
When we consider a group $G$ and say that \textit{$G$ acts on a Riemann surface $S$}, we are saying  that there exists a group  monomorphism from $G$ to $\auto(S)$, where $\auto(S)$ is the group consisting of the self--maps of $S$ (automorphism or bi--holomorphic map) which preserve the  complex structure.

\medskip
\noindent Two subgroups, say $H_{0}\, ,H_{1} < Aut(S )$ are said to be \textit{conformally equivalent} (respectively,
\textit{topologically equivalent}) if there exists an  automorphism (respectively, an homeomorphism) $t : S \rightarrow S$ so that $tH_{0}t^{-1} = H_{1}$.

\noindent It is clear from the definition that any two conformally equivalent subgroups are topologically
equivalent, but the reciprocal is in general false.

\noindent G. Gonz\'alez in \cite{gab1,gab2} proved that if $H_0$ and $H_1$ are cyclic groups of order $p$ prime, $S/H_{0}$ is
the Riemann sphere and $H_{0}$ and $H_1$ are topologically equivalent, they should be conformally
equivalent.

\noindent Continuing with the cyclic case  for a group of prime order, a relationship between two topologically equivalent actions for the generating vectors is given by J.\,Gilman in \cite{gilman}.
  The case where the group is cyclic,  a relationship between the local structure for the automorphisms with fixed points and the epimorphism associated to the action is given by W. Harvey in \cite{harvey}.

\noindent Later in \cite{gabrub}, G.\,Gonz\'alez--Diez and R.\,Hidalgo give an example of two actions of $\mathbb{Z}/8\mathbb{Z}$ on a family  of compact Riemann surfaces of genus $9$ that are directly topologically, but not conformally, equivalent, except for  finitely many cases.

 \medskip
\noindent
Studying the classification  of actions contributes to the understanding of the properties of the moduli space $\mathcal{M}_{g}$.

\noindent For a compact Riemann surface $S_{0}$ of genus $g$, consider the subgroup $H_{0}\leq \auto(S_{0})$,  the set
$$X(S_{0},H_{0})=\left\{(S,H):\exists t \in \home^{+}(S_{0},S)\, , tH_{0}t^{-1}=H \right\}$$
and the equivalence relation: $(S_{1},H_{1}) \sim (S_{2},H_{2})$ if and only if there is
$\phi \in \isom(S_{1},S_{2})$ so that $\phi H_{1} \phi^{-1}=H_{2}$.

\noindent We denote by $\widetilde{\mathcal{M}}_{g}(H_{0})$ the quotient space defined by the above relation. This turns out to be a normal space.

\noindent Consider $\mathcal{M}_{g}$ the moduli space associated to $S_{0}$, that is, a model of moduli space of genus $g$.

\noindent Let $\mathcal{M}_{g}(H_{0})=\{[S]\in \mathcal{M}_{g}\, : \exists t \in \home^{+}(S_{0},S)\, , tH_{0}t^{-1}<\auto(S)\}$.

\noindent The forgetful map is defined by
\begin{eqnarray*}
\mathcal{P}:\widetilde{\mathcal{M}}_{g}(H_{0}) &\longrightarrow& \mathcal{M}_{g}(H_{0}) \\
\text{ } [(S,H)] &\rightsquigarrow& [S]
\end{eqnarray*}
As is well known,
$\widetilde{\mathcal{M}}_{g}(H_{0})$ is the normalization of $\mathcal{M}_{g}(H_{0})$.
Moreover, $\mathcal{P}$ is not bijective if only if there exists a compact Riemann surface $S$ of genus $g$ admitting two groups of automorphisms $H_{1}$ and $H_{2}$ which are directly topologically, but not conformally, conjugate to $H_{0}$.
For further details, see \cite{gabhar}.

\medskip
\noindent Section 2 contains an overview of definitions and relevant results about automorphisms of Riemann surfaces and Fuchsian groups.

\medskip
\noindent  Section 3  contains some our contribution to the problem of the classification of actions. For cyclic groups,
Theorem~\ref{topocyc} gives a condition on the generating vectors under which two actions are directly topologically equivalent.
Also, we generalize a result due to Harvey \cite[Theorem~7]{harvey}.

\medskip
\noindent Section 4,  inspired by the paper of G.\,Gonz\'alez--Diez and R.\,Hidalgo~\cite{gabrub}, we produce for each $n \in \mathbb{N}$  the families $\mathfrak{S}_{1}$ and $\mathfrak{S}_{2}$. By definition, $\mathfrak{S}_{i}$ with $i=1,2\,, \,$ consists of the Riemann surfaces of genus $3(2^{n}-1)$ defined by
\begin{equation*}
    f_{a,\lambda}\left(x,y\right)=y^{2^{n}}-x^{a}\left(x^{2}-1\right)^{a}\left(x^{2}-\lambda^{2}\right)\left(x^{2}-\lambda^{-2}\right)
    \end{equation*}
When $i=1$ (resp. $i=2$), the automorphism group for the elements of $\mathfrak{S}_{1}\,$ (resp. $\mathfrak{S}_{2}$) is $\, \mathbb{Z}/2^{n+1}\mathbb{Z} \times \mathbb{Z}/2\mathbb{Z}\,$ (resp. $\mathbb{Z}/2^{n+1}\mathbb{Z} \rtimes_{h} \mathbb{Z}/2\mathbb{Z}$).
In both cases,  there exist two cyclic subgroups which define directly topologically, but not conformally,  equivalent actions.

\section{Preliminaries}


\medskip
\noindent We say $\sigma$ is a \textit{holomorphic map of Riemann Surfaces}  from $S_{1}$ to $S_{2}$  if, for each $P\in S_{1}$, $(U,\phi)$ chart centered at $P$ and  $(V, \psi)$  chart centered at $Q=\sigma(Q)\in S_{2}$, then we have $\psi \circ \sigma \circ \phi^{-1}$ is a holomorphic function. The order at zero for this holomorphic function it is called the \textit{multiplicity  } at $P$.
\noindent When the multiplicity at $P$ is greater than or equal $2$ we say $P$ is a \textit{ramification point of} $\sigma$. The point $Q=\sigma(P)$ is called \textit{branch point of} $\sigma$.

\medskip
\noindent For a bijective holomorphic map, $\sigma:S_{1} \longrightarrow S_{2}$,  we say $\sigma$ is a \textit{bi-holomorphic map} or a \textit{isomorphism} between Riemann surfaces.  When $S_{1}=S_{2}$ we say $\sigma$ is an \textit{automorphism of} $S$.
From now on, $\isom(S_{1},S_{2})$ (respectively $\auto(S)$) denotes the isomorphisms set between $S_{1}$ and $S_{2}$ (respectively automorphism of $S$).

\noindent Let $S$ be  a  compact Riemann surface of genus $g\geq2$. For each $P \in S$, consider the subgroup of $\auto(S)$ given by
    \begin{equation*}
    \auto(S)_{P}=\left\{\sigma \in \auto(S):\sigma(P)=P\right\}\, ,
    \end{equation*}
 called \textit{stabilizer subgroup}.

\noindent Now let $(U,\phi)$ be a chart centered at $P$, and  $\sigma \in \auto(S)_{P}$ then we have
    \begin{equation*}
    \phi \circ \sigma \circ \phi ^{-1}(z)=\sum_{m\geq 1}c_{m}(\sigma)z^{m}
    \end{equation*}
and we define
    \begin{eqnarray*}
    \delta_{P} :  \auto(S)_{P} & \longrightarrow & S^{1}=\{z\in \mathbb{C}:|z|=1\} \\
         \sigma & \rightsquigarrow & c_{1}(\sigma)
    \end{eqnarray*}
\begin{theo} \label{delta}
    The map $\delta_{P}$ is a group monomorphism. Further,  $\auto(S)_{P}$ is a cyclic finite subgroup of $\auto(S)$.
    \end{theo}
The proof of this theorem can be found in \cite{miranda}.


\noindent Note that for $\tau \in \auto(S)$ and $\sigma\in \auto(S)_{P}$ we have $\tau \circ \sigma \circ \tau^{-1} \auto(S)_{q}$ where $Q=\tau(P)$. It is not difficult prove that $\delta_{Q}(\tau \circ \sigma \circ \tau^{-1})=\delta_{P}(\sigma)$.



\subsection{Fuchsian groups.}
Let $\Delta$ denote the unit disk $\{z : |z|<1\}$ and let $\auto(\Delta)$ be the group of M\"obius transformations self--mappings of $\Delta$. A \textit{Fuchsian group} is a discrete subgroup $\Gamma$ of $\auto(\Delta)$.
\noindent Let $PSL(2,\mathbb{C})$ denote the M\"obius transformations group and let $T$ be a M\"obius transformation. If $T\neq 1$ then $T$ has one or two fixed point. If $T$ has one fixed point, it is called \textit{parabolic} transformation. Now let $\Upsilon$ be a subgroup of $PSL(2,\mathbb{C})$. We say that $\Upsilon$ acts \textit{properly discontinuously} at $z_{0} \in \widehat{\mathbb{C}}$ provided that the stabilizer subgroup of $\Upsilon$ at $z_{0}$, $\Upsilon_{z_{0}}$, is finite, and there exists a neighborhood $U$ of $z_{0}$ such that
$$T(U)=U \qquad ,\forall T \in \Upsilon_{z_{0}} \qquad\text{ and } \quad  U \cap T(U)=\emptyset \qquad ,\forall T \in \Upsilon - \Upsilon_{z_{0}} $$
\medskip
\noindent Denote  by $\Omega(\Upsilon)$ the \textit{region of discontinuity}  of $\Upsilon$, that is, the set of point $z_{0}\in \widehat{\mathbb{C}}$ such that $\Upsilon$ acts properly discontinuously at $z_{0}$. The complement of set $\Omega(\Upsilon)$ is denote by $\Lambda(\Upsilon)$ and call the \textit{limit set} of $\Upsilon$.

\noindent A Fuchsian group $\Gamma$ (acting on the unit disk $\Delta$) must satisfies that $\Delta \subset \Omega(\Gamma)$.

\bigskip
\noindent Let $\Gamma$ be a Fuchsian group. Since $\Gamma$ acts on $\Delta$, we may consider the natural projection
$$\pi_{\Gamma}:\Delta \longrightarrow \Delta/\Gamma\, $$

$\mathcal{O}=\Delta/\Gamma$ has a \textit{Riemann orbifold structure}, that is,
        \begin{itemize}
        \item[(i)] an underlying Riemann surface structure $\mathcal{O}$ so that $\pi_{\Gamma}: \Delta \to \mathcal{O}$ is a
        holomorphic map;
        \item[(ii)] a discrete collection of cone points (branch  points of $\pi$); and
        \item[(iii)] at each cone point $p$ a cone order; this being the order of the stabilizer cyclic
        subgroup of any point $q$ so that $\pi(q)=p$.
        \end{itemize}

\noindent If $\Gamma$ is torsion--free such that  $\Lambda(\Gamma)=S^{1}$ then $\Pi_{1}\left(\mathcal{O}\right) \cong \Gamma$, and $\mathcal{O}$ has not cone point.
Furthermore let $\Gamma'$  be a Fuchsian groups such that $\Gamma'$ is torsion--free and $\Lambda(\Gamma')=S^{1}$.
Then $S=\Delta/ \Gamma $  and $S'=\Delta/\Gamma'$ are isomorphic Riemann surfaces  if only if there exists $T \in \auto(\Delta)$ such that $\Gamma'=T \Gamma T^{-1}$.

\noindent If $\Gamma$ is finitely generated, without parabolic transformations and $\Lambda(\Gamma)=S^{1}$, then $\mathcal{O}$ is a
compact Riemann surface of some genus $\gamma$ and there are a finite set of cone points. Furthermore if $\Gamma$ torsion--free then the genus of $\Delta/\Gamma$ is at less $2$.


\noindent If $\Gamma$ is finitely generated, without parabolic transformations and
$\Lambda(\Gamma)=S^{1}$, whose underlying Riemann surface has genus
$\gamma$ and the cone orders are $m_{1},\cdots ,m_{r}$, then we define its \textit{signature}  (for both,
$\Gamma$ and ${\mathcal O}$) as the tuple $(\gamma;m_{1},\cdots,m_{r})$.

\noindent The holomorphic map $\pi_{\Gamma}$ is called a \textit{branched covering}  of type $(\gamma;m_{1},\cdots ,m_{r})$.

\noindent When a Fuchsian group $\Gamma$ has  signature $\left(\gamma;m_{1},\cdots , m_{r}\right)$,
there is a presentation associated for the group $\Gamma$, this is, there exist $a_{1},b_{1}, \cdots, a_{\gamma},b_{\gamma},x_{1},\cdots,x_{r}\in \Gamma$ such that $\Gamma$ has a presentation:
            \begin{equation}\label{presentation}
            \Gamma = \left<a_{1},b_{1}, .. \, , a_{\gamma},b_{\gamma},x_{1},..\,,x_{r}: x_{1}^{m_{1}}=\cdots =x_{r}^{m_{r}}=\prod_{i=1}^{\gamma}[a_{i},b_{i}]\prod_{j=1}^{r}x_{j}=1 \right>
            \end{equation}
where $[a_{i},b_{i}]=a_{i}b_{i}a_{i}^{-1}b_{i}^{-1}$.

\noindent Further, we have that for each $j$,  the subgroup generate for $<x_{j}>$ is a maximal finite cyclic subgroup. Moreover, this subgroup is the $\Gamma-$ stabilizer of a unique point in $\Delta$, and each element of finite order in $\Gamma$ is conjugate to a power of some $x_{j}$.

\noindent When $x_{j}$ is conjugate (in $\auto(\Delta)$) to the rotation $R(z)=\exp\left(\frac{2\pi i}{m_{j}}\right)z$ (respectively $R(z)=\exp\left(-\frac{2\pi i}{m_{j}}\right)z$), say $x_{j}$ is a \textit{positive minimal rotation} (respectively \textit{non--positive minimal rotation}). We say that $x_{j}$ is a \textit{minimal rotation} in any these cases.

\noindent According to the paper of L.Keen \cite{keen} (also see \cite[Theorem 4.3.2]{katok}) we may construct for $\Gamma$ a hyperbolic polygon with $4\gamma+2r$ sides.


\noindent Associated to this hyperbolic polygon we may find a presentation for $\Gamma$ as \eqref{presentation} such that for each $j=1,\,..\,,r$ \, $x_{j}$ is a positive minimal rotation.

\noindent We remark that for every  signature $\left(\gamma;m_{1},m_{2},\cdots , m_{r}\right)$ such  that
            $$2\gamma -2 + \sum\left(1-\dfrac{1}{m_{j}}\right)> 0$$
there exists $\Gamma$ a Fuchsian group, uniquely determined up conjugation in $\auto(\Delta)$, with this signature.

\noindent For more details see \cite{macbeath},  \cite{farkas}, \cite{katok} and \cite{maskit}.

\bigskip

\noindent Consider  two Fuchsian groups $\Gamma_{1}, \Gamma_{2}$ . We  say that $\Gamma_{1}$ is \textit{geometrically isomorphic} to $\Gamma_{2}$ if there exists a self--homeomorphism of $\Delta$, say $T \in \home (\Delta)$, and a group isomorphism $\chi:\Gamma_{1} \rightarrow \Gamma_{2}$ such that for all $x\in \Gamma_{1}$ the following holds
            \begin{equation*}
            \chi(x)=T \circ x \circ T^{-1} \, .
            \end{equation*}

\noindent We  also say that the group isomorphism  $\chi:\Gamma_{1} \rightarrow \Gamma_{2}$  can be \textit{realized geometrically} if there exists $T \in \home (\Delta)$ such that the previous condition is true.

\begin{theo}\label{nec}
Let $\chi:\Gamma_{1}  \longrightarrow \Gamma_{2}$ be an isomorphism between finitely generated Fuchsian group, both without parabolic elements, with $\Lambda(\Gamma_{j})=S^{1}$, for $j=1,2$. Then $\chi$ is geometric.
\end{theo}

\noindent The preceding theorem holds at the level of Non Euclidean plane Crystallographic groups (NEC groups), that is, finitely generated discrete subgroup of isometries of the hyperbolic disc containing no parabolic elements.
\noindent For more details see \cite[pag. 1201]{macbeath}.

\noindent At this point it is important to note that two homeomorphism, say $F_{1},F_{2}:\Delta \longrightarrow \Delta$, defining the same isomorphism $\chi:\Gamma_{1} \longrightarrow \Gamma_{2}$, must have the same orientability  type. In fact, the homeomorphism $F_{2}^{-1}\circ F_{1}:\Delta \longrightarrow \Delta$ defines the identity automorphism of $\Gamma_{1}$. It can be proved that, in this case, $F_{2}^{-1}\circ F_{1}$ is homotopic to the identity.

\noindent We remark that in general  we may not assume the homeomorphism that realizes the isomorphism should be orientation preserving. An example of this situation is:   let $\Gamma$ be a Fuchsian group with signature $(0; 5,5,5)$ the isomorphism $\chi:\Gamma \rightarrow \Gamma$ given by $\chi(x_{j})=x_{j}^{-1}$ (for $j=1,2$ and notations as \eqref{presentation}).

\bigskip

\bigskip

\subsection{Actions on Riemann surfaces.}

Let $\varepsilon:G\rightarrow \auto(S)$ be an action of $G$ on $S$, where $S$ is a Riemann surfaces of genus at less $2$ and $G$ is a finite group.

\noindent We may consider for each $P\in S$ the \textit{stabilizer subgroup}  of $P$, this is
$$G_{P}=\left\{D \in G : \varepsilon(D)(P)=P \right\}\, .$$
Since $\varepsilon(G_{p})\leq \auto(S)_{P}$ it follow of the theorem \eqref{delta} that  $G_{P}$ is a cyclic group.

\noindent For the  action $\varepsilon:G \rightarrow \auto(S)$ we have the natural projection $\pi: S \longrightarrow S/ \varepsilon(G)$.
Using this projection we can give to $S/ \varepsilon(G)$ an  complex structure. Hence $S/ \varepsilon(G)$ is a compact Riemann surfaces and $\pi$ is a holomorphic map of Riemann surfaces.

\noindent Further we have that $P\in S$ is a ramification point of $\pi$ if only if $G_{P}\neq \{Id\}$, furthermore the multiplicity of $P$ is $|G_{P}|$. Then $\pi$ is a  smooth covering (unbranched covering) on the complement of a finite set, the ramification points set.

\noindent We called to $\pi$  a \textit{branched covering}, and we say that $\pi$ has a \textit{signature} $(\gamma; m_{1},\cdots ,m_{r})\, ,$
where $\gamma$ is the genus of $S/ \varepsilon(G)$\, , $m_{j}$ are the multiplicity of the ramification points, and $r$ is the number of the branch points of $\pi$. Sometimes also we will say $G$ acts on $S$ with signature $(\gamma; m_{1},\cdots ,m_{r})$. For this notation we suppose $m_{1}\geq \cdots \geq m_{r}\, .$

\noindent The following theorem give a relation between the action on Riemann surfaces theory and the Fuchsian group theory:
\begin{theo}\label{existence2} \label{existence}
    Let $S$ be a compact Riemann surface of genus $ g \geq 2$ and let $G$ be a finite group.
    There is an action $\varepsilon$  of $G$ on $S$  with signature $(\gamma;m_{1},\cdots,m_{r})$ if and only if there
    are a Fuchsian group $\Gamma$ with signature $( \gamma;m_{1},\cdots,m_{r})$ , an epimorphism $\theta_{\varepsilon}:\Gamma \rightarrow G$ such that $K=\ker(\theta_{\varepsilon})$ is torsion--free Fuchsian group, $\Lambda(K)=S^{1}$ and $\Delta/K$
    is a Riemann surface isomorphism to $S$.
    \end{theo}
    We have $\Delta/K$ is a  Riemann surface isomorphic to $S$. Call $f: \Delta/K \rightarrow S$ to this isomorphism and $\pi_{K}:\Delta \rightarrow \Delta/K$ to the natural projection. If $X$ is the universal covering for $S$, then we may lift $f$ to $F:\Delta \rightarrow X$ isomorphism of Riemann surfaces.

    \noindent For $x\in \Gamma$ we define $\theta_{\varepsilon}(x)=f\circ \widetilde{x} \circ f^{-1}$ where $\widetilde{x}$ is the automorphism of $\Delta/K$ induced by $x$ (i.e. $\pi_{K} \circ x = \widetilde{x}\circ \pi_{K}$).


\section{Equivalence of actions}

\noindent Let $S_{j}$ be a Riemann Surface for $j=1,2$, and $G$ be a group.
\noindent The actions  $\varepsilon_{1}$, $\varepsilon_{2}$ of $G$ on $S_{1}$ and $S_{2}$ respectively,  are called \textit{topologically equivalent} if there exist $\Phi \in \auto{(G)}$ and $t \in \home{(S_{1}, S_{2})}$ such that the following diagram is commutative
             \[\begin{diagram}
            \node{G}\arrow{e,t}{\varepsilon_{1}} \arrow{s,e}{\Phi} \node{\varepsilon_{1}(G)} \arrow{s,r}{\Psi_{t}} \\
            \node{G} \arrow{e,t}{\varepsilon_{2}} \node{\varepsilon_{2}(G)}
            \end{diagram}\]
where $\Psi_{t}(\tau):= t\, \circ \, \tau \, \circ t^{-1}$ and  $\home{(S_{1} , S_{2})}$ is the group of homeomorphisms of $S_{1}$ on $S_{2}$.

\noindent If $t\in \home^{+}(S_{1}, S_{2})$ then we say that $\varepsilon_{1}$, $\varepsilon_{2}$ are \textit{directly topologically equivalent}, where $\home^{+}{(S_{1}, S_{2})}$ is the group of orientation preserving homeomorphisms of $S_{1}$ on $S_{2}$. Further if
$t\in \isom(S_{1},S_{2})$ then we say that $\varepsilon_{1}$, $\varepsilon_{2}$ are \textit{conformally equivalent}.

\bigskip

\noindent If  $\varepsilon_{j}$ is an action of $G$ on $S_{j}$, then by the theorem \ref{existence2} there exist  $\Gamma_{j}$ and $K_{j}$ Fuchsian groups and an epimorphism $\theta_{\varepsilon_{j}}=\theta_{j}:\Gamma_{j} \rightarrow G$ such that $K_{j}=\ker(\theta_{j})$. The following theorem we give a relation between $\theta_{1}$ and $\theta_{2}$ when the actions $\varepsilon_{1}$ and $\varepsilon_{2}$ are equivalents of some type.

\begin{theo}\label{topo}\label{conf}
    \noindent $\varepsilon_{1}$ is topologically equivalent to $\varepsilon_{2}$ (respectively directly topologically equivalent or conformally equivalent) if only if there exists $T \in \home(\Delta)$ (respectively $T \in \home^{+}(\Delta)$ or $T\in \auto(\Delta)$) and a group isomorphism $\Phi:\varepsilon_{1}(G) \rightarrow \varepsilon_{2}(G)$ such that $\Phi \, \circ \, \theta_{1} = \theta_{2} \, \circ \, \chi_{T}$
    where $\chi_{T}(x)=T\, \circ\, x \, \circ \, T^{-1}$.

   \end{theo}

\begin{proof}
For each $j$, call $f_{j}: \Delta/K_{j} \rightarrow S_{j}$ to the isomorphism between $\Delta/K_{j}$ and $S_{j}$ given by the theorem \ref{existence2}. If $\varepsilon_{1}$ is topologically equivalent to $\varepsilon_{2}$ (respectively directly topologically equivalent or conformally equivalent) then there exists $t\in \home(S_{1},S_{2})$ (respectively $t\in \home^{+}(S_{1},S_{2})$ or $t\in \isom(S_{1},S_{2})$). Now we may lift $f_{2}^{-1}\circ t \circ f_{1}$ to $T \in \home(\Delta)$ (respectively $T \in \home^{+}(\Delta)$ or $T \in \auto(\Delta)$). For $x\in \Gamma_{1}$, according to the notation of the theorem \ref{existence2},  we have the following diagram:
\[\begin{diagram}
            \node{\Delta} \arrow{e,t}{T^{-1}} \arrow{s,l}{} \node{\Delta} \arrow{e,t}{ x} \arrow{s,l}{} \node{\Delta} \arrow{s,l}{ } \arrow{e,t}{T} \node{\Delta}  \arrow{s,l}{}\\
            \node{\Delta/K_{2}} \arrow{e,t}{f_{1}^{-1}t^{-1}f_{2}} \node{\Delta/K_{1}} \arrow{e,t}{f_{1}^{-1}\theta_{1}(x)f_{1}} \node{\Delta/K_{1}} \arrow{e,t}{f_{2}^{-1}tf_{1}} \node{\Delta/K_{2}}
            \end{diagram}\]

\noindent Note that $T \,\circ \, x \, \circ \, T^{-1} \in \Gamma_{2}$. Furthermore if  $x\in K_{1}$ then $T \,\circ \, x \, \circ \, T^{-1} \in K_{2}$.

\noindent Using the diagram we have $\theta_{2}(T\circ x \circ T^{-1})=\Psi_{t}(\theta_{1}(x))$.

\bigskip

\noindent Reciprocally we have $\Phi \, \circ \, \theta_{1} = \theta_{2} \, \circ \, \chi_{T}$ then $\chi_{T}(K_{1})=K_{2}$. Therefore $T \in \home(\Delta)$ (respectively $T\in \home^{+}(\Delta)$ or $T\in \auto(\Delta)$ )  define  $t\in \home(S_{1},S_{2})$ (respectively $t\in  \home^{+}(S_{1},S_{2})$ or $t\in \isom(S_{1},S_{2})$) such that $\Phi=\Psi_{t}$ and the actions are topologically equivalents (respectively directly topologically equivalent or conformally equivalent).
\end{proof}

\noindent Remark that for $S_{1}=S_{2}=S$ the actions $\varepsilon_{1}$ and $\varepsilon_{2}$  are topologically equivalents (respectively directly topologically equivalent or conformally equivalent ) if only if $\varepsilon_{1}(G)$ and $\varepsilon_{2}(G)$ are conjugate group in $\home(S)$ (respectively in $\home^{+}(S)$ or $\auto(S)$).

\subsection{Cyclic groups}
 Let $G=\mathbb{Z}/n\mathbb{Z}$ be a cyclic group of order $n$. We consider $G$ as the integers module $n$ ($G=\{0,1,\cdots,n-1\}\,$) and $\Gamma$ a Fuchsian group with signature $(\gamma;m_{1},\cdots , m_{r})$ and presentation as \eqref{presentation}.

\noindent The following lemma was given by Kuribayashi (see \cite[Lemma 3.1]{kuribayashi}).

    \begin{lem}
    Let $\theta:\Gamma \longrightarrow G$ be a group epimorphism and assume that $\nu$ is a generator of $G$. Then for any permutation $\mu$ of $\{1,\cdots,r\}$ with $m_{\mu(j)}=m_{j}\,$ ($j=1,\cdots,r$), there exists an automorphism $\chi$ of $\Gamma$ such that
        \begin{itemize}
        \item[(i)] $\theta \circ \chi(a_{i})=\theta \circ \chi(b_{i})=\nu$\, \qquad  with $i=1,\cdots,\gamma$;
        \item[(ii)]$\chi(x_{j})=D_{j}x_{\mu(j)}D_{j}^{-1}$, \qquad for some $D_{j}\in \Gamma$\, with $j=1,\cdots,r.$
        \end{itemize}
    \end{lem}

\noindent Remark that in the article of Kuribayashi  the automorphisms of $\Gamma$ used to computed the automorphism $\chi$ are geometric induced by orientation preserving homeomorphism.

\noindent Now on, no loss of generality consider the group epimorphism $\theta$ given by
    \begin{align}\label{kuri}
    \theta : \Gamma & \longrightarrow  G \\
     a_{i} & \rightsquigarrow  1 \qquad, \text{ with } i=1,..,\gamma \notag \\
     b_{i} & \rightsquigarrow  1  \qquad, \text { with }i=1,..,\gamma \notag \\
    x_{j} & \rightsquigarrow  \xi_{j} \qquad, \text { with } j=1,..,r \notag
    \end{align}
where $K=\ker(\theta)$ is a torsion free group.

\noindent Hence, it follow
\begin{theo}\label{topocyc}
Let $\Gamma$, $\Gamma'$ Fuchsian groups both with signature $(\gamma;m_{1},\cdots , m_{r})$ and presentation as \eqref{presentation} according to your hyperbolic polygon associated. Consider the group epimorphisms $\theta$, $\theta'$ as \eqref{kuri}.
\noindent If there exists an $s \in \mathbb{Z}$ with $(s,n)=1$, such that
            \begin{equation} \label{congruence}
            (\xi'_{1},...,\xi'_{r})\equiv s(\xi_{1},...,\xi_{r}) ,\mod n
            \end{equation}
then the actions induced by  $\theta$ on $\Delta/K$ and by $\theta'$ on   $\Delta/K'$  are directly topologically equivalent.
\end{theo}
 \begin{proof}
 \noindent  Suppose that we have the equation \eqref{congruence}, then we may define the homomorphism $\Phi: \mathbb{Z}/n\mathbb{Z} \longrightarrow \mathbb{Z}/n\mathbb{Z}\, ,$ given by $\Phi(1)=s$. Then for each $j$ we have
        $$\Phi(\xi_{j})=s\xi_{j}=\xi'_{j}\, .$$

\noindent Since $s$ and $n$ are relative primes,  $\Phi$ is an automorphism.

\noindent We consider  the following diagram
        \[\begin{diagram}
        \node{0} \arrow{e,t}{ } \node{K} \arrow{s,r,..}{ } \arrow{e,t}{ }
        \node{\Gamma} \arrow{s,r,..}{ \chi } \arrow{e,t}{\theta} \node{\mathbb{Z}/n\mathbb{Z}} \arrow{e,t}{ } \arrow{s,r}{\Phi} \node{0} \\
        \node{0} \arrow{e,t}{ } \node{K'} \arrow{e,t}{ } \node{\Gamma'} \arrow{e,t}{\theta'} \node{\mathbb{Z}/n\mathbb{Z}} \arrow{e,t}{ }\node{0}
        \end{diagram}\]
 hence we may define
        \begin{equation*}
        \chi(a_{j})=a'_{i}, \qquad  \chi(b_{i})=b'_{j}, \qquad
        \chi(x_{j})=x'_{j}
        \end{equation*}
 and then we have that the diagram is commutative.

   \noindent Since $\chi$ maps generators on generators, and these elements satisfy the same relation, we have $\chi$ is an isomorphism. Further by the commutative diagram $\chi(K)=K'$ therefore $\chi|_{K}:K\rightarrow K'$ is an isomorphism.

    \noindent It follow of Theorem \ref{topo} that the actions induced by $\theta$ and  by $\theta'$ are topologically equivalent.

    \noindent Our next claim is that the actions induced by $\theta$ and $\theta'$ are directly topologically equivalent. We have to construct according to \cite{keen} the hyperbolic polygons associated $\Gamma$ and $\Gamma'$.

    \noindent Since $x_{j}$ and $x'_{j}$ are positive minimal rotation in the same angle, we have the automorphism $\chi$ is induced by a $f\in \home^{+}(\Delta)$.
    \end{proof}

\noindent Reciprocally consider $\varepsilon$ and $\varepsilon'$ actions on $S$ and $S'$, respectively, such that the signature for the actions is $(\gamma; m_{1},\cdots, m_{r})$. If $\varepsilon$ is topologically equivalent to $\varepsilon'$ then (according to the theorem \ref{topo}) there exists $T \in \home(\Delta)$ such that $ \Psi_{t}\circ \theta_{\varepsilon}= \theta_{\varepsilon'} \circ \chi_{T}$.
\noindent Call $\theta=\varepsilon^{-1}\circ \theta_{\varepsilon}$ and $\theta'=(\varepsilon')^{-1}\circ \theta_{\varepsilon'}$. Then there exists $\Phi\in \auto\left(\mathbb{Z}/n\mathbb{Z}\right)$ such that $\Phi \circ \theta = \theta' \circ \chi_{T}$. Therefore
$\theta' \left(\chi_{T}(x_{j})\right)= \Phi(1) \theta(x_{j})$.

\noindent Observe that this result  is a generalization of a result of J. Gilman in which $n$ is a prime number (See \cite[Lemma 2, p.\,54]{gilman}.)

\subsection{Epimorphism and local structure}

\noindent In \cite{harvey} Harvey gives a relation between  cyclic covering of Riemann sphere  and the epimorphisms of Fuchsian group using the rotation angles. Our result gives a relation for any covering of Riemann sphere.

\noindent As we have seen, for an action  of a finite group $G $ on a compact Riemann surface $S$ (no loss generality consider $G<\auto(S)$),  according to  theorem \ref{existence2} we have an epimorphism $\theta: \Gamma \longrightarrow G\, ,$ defined by $\theta(x) = f \circ \widetilde{x} \circ f^{-1}$.  Furthermore $K=\ker(\theta)$ is a torsion-free Fuchsian group and $\Gamma$ is a Fuchsian group with the same  signature that the action, and $S$ is isomorphic to $\Delta/K$.

\noindent The following theorem we give a relation between the epimorphism $\theta$, and the homomorphism $\delta_{P}$ , for $P$ fixed point of the action.
\begin{theo}
    Let  $\sigma \in \auto(S)_{P} \leq \auto(S)$ of order $n$, and let
        \begin{equation*}
         \mathcal{L}_{P}=\left\{x \in \auto(\Delta):\,
         \exists z_{0} \in \Delta\, , f\circ \pi_{K}(z_{0})=P\, , x(z_{0})=z_{0} \text{ and }
         (f^{-1} \circ \sigma \circ f )\circ \pi_{K} = \pi_{K} \circ x  \right\}
        \end{equation*}
    \noindent Then there is unique primitive complex $n$th root of unity $\zeta$ such that for all $x\in \mathcal{L}$ we have that  $x$ is conjugate to multiplication by $\zeta$, $R(z)=\zeta z$, in $\auto(\Delta)$. Furthermore  $\zeta=\delta_{P}(\sigma)$.
    \end{theo}

    \noindent The proof of this theorem can be found in \cite{breuer}.

    \noindent Since $\zeta$ is a primitive complex $n$th root of unity,  we may write $\zeta=\omega_{n}^{j}$  where
    $\omega_{n}=\exp\left(\frac{2\pi i}{n}\right)$ and the numbers $j$ and $n$ are  relative primes ($(j,n)=1$). We call $\frac{2\pi j i}{n}$ the \textit{rotation angle}  for $\sigma$ at $P$.

\bigskip



\noindent Now we consider $\Gamma$ with presentation as \eqref{presentation} according to your hyperbolic polygon associated. Recall $x_{j}$ is a positive minimal rotation. Consider $z_{j}$ the fixed point of $x_{j}$. Then $x_{j}\in \mathcal{L}_{P_{j}}$ where $P_{j}=f(\pi_{K}(z_{j}))$.

\noindent Furthermore by the preceding theorem we have $\delta_{P_{j}}(\theta(x_{j}))=\omega_{m_{j}}$.

\begin{theo}\label{harvey}
    Let $G$ be a subgroup of $\auto(S)$, where $G$ acts on $S$ with signature $(0;m_{1},\cdots , m_{r})$. The epimorphism $\theta$ is determined by the fixed points of the action and their stabilizer groups. In other words, if we consider $x_{j}$ and $P_{j}$ as before, then
    $$\theta(x_{j})=\tau_{j}^{\xi_{j}}$$
    where $<\tau_{j}>=G_{P_{j}}$, and where the number $\xi_{j}$ is determined by the equations
        \begin{eqnarray*}
        \delta_{P_{j}}(\tau_{j})&=&\omega_{m_{j}}^{\eta_{j}}\, , \qquad  1\leq\eta_{j}<m_{j}, (\eta_{j},m_{j})=1,\\
        \eta_{j}\cdot \xi_{j} &\equiv& 1 \mod m_{j}
        \end{eqnarray*}
    where $1\leq\xi_{j}<m_{j}$, with $(\xi_{j},m_{j})=1$.
    \end{theo}

    \begin{proof}





    \noindent Recall that for the fixed point $P_{j}$ we have $G_{P_{j}}$ is a cyclic group of order $m_{j}$.
    Let $\tau_{j}$ be a generator of $G_{P_{j}}$. Then  $\theta(x_{j})=\tau_{j}^{\xi_{j}}$ for some $0<\xi_{j}<m_{j}$ this is because $\theta(x_{j})\in G_{P_{j}}$.

    \noindent Since $\delta_{P_{j}}$ is a group monomorphism and $\delta_{P_{j}}\left(\theta(x_{j})\right)=\omega_{m_{j}}$ then
    $$\omega_{j}= \delta_{P_{j}}\left(\tau_{j}^{\xi_{j}}\right)= \delta_{P_{j}}(\tau_{j})^{\xi_{j}}$$
    therefore $\left(\xi_{j},m_{j}\right)=1$.

    \noindent If $\delta_{P_{j}}(\tau_{j})=\omega_{m_{j}}^{\eta_{j}}$ (where $(\eta_{j},m_{j})=1$) then
    $$\omega_{m_{j}}=\delta(\tau_{j})^{\xi_{j}}=\omega_{m_{j}}^{\eta_{j}\, \xi_{j}}$$
    hence
    $$1\equiv \eta_{j}\, \xi_{j} \quad \mod{m_{j}}$$

    \noindent  If we take another generator of $G_{P_{j}}$, say $\widehat{\tau}_{j}$, then we may to do the same computations, hence
    $$\theta(x_{j})={\widehat{\tau}_{j}}^{\widehat{\xi}_{j}}$$
    where
    \begin{eqnarray*}
        \delta_{P_{j}}({\widehat{\tau}}_{j})&=&\omega_{m_{j}}^{\widehat{\eta}_{j}}\\
        {\widehat{\eta}}_{j}\cdot \widehat{\xi}_{j} &\equiv& 1 \mod m_{j}
        \end{eqnarray*}

    \noindent As $\tau_{j}$ is a generator of $G_{P_{j}}$, we have there exists $0<t<m_{j}$ such that $\widehat{\tau}_{j}=\tau_{j}^{t}$ thus we have

    $$\omega_{m_{j}}^{\widehat{\eta}_{j}}=\delta_{P_{j}}(\widehat{\tau}_{j})=\delta_{P_{j}}(\tau_{j}^{t})
    =\omega_{m_{j}}^{\eta_{j}\cdot t}$$
    then
        \begin{eqnarray*}
        \widehat{\eta}_{j} &\equiv& \eta_{j}\cdot t \mod m_{j} \qquad |\cdot \xi_{j} \\
        \widehat{\eta}_{j}\cdot \xi_{j} &\equiv& \eta_{j}\cdot \xi_{j}t \equiv  t  \mod m_{j} \qquad | \cdot \widehat{\xi}_{j}\\
        \xi_{j} \equiv  \widehat{\xi}_{j}\cdot \widehat{\eta}_{j}\cdot \xi_{j} &\equiv& t\cdot \widehat{\xi}_{j}  \mod m_{j}
        \end{eqnarray*}
    therefore $\tau_{j}^{\xi_{j}} = {\widehat{\tau}_{j}}^{\widehat{\xi}_{j}}$.
    \end{proof}

    \noindent Remark that as $\theta$ is an epimorphism for the $r-$tuple $(\theta(x_{1}),\cdots, \theta(x_{r}))$ is has
    \begin{enumerate}
        \item $G=<\theta(x_{1}),\cdots, \theta(x_{r})>$.
        \item $\ord(\theta(x_{j}))=m_{j}$ , for each $j=1,..\,,r$.
        \item $\theta(x_{1})\cdots \theta(x_{r})=1$.
        \end{enumerate}

    \noindent In general we say that the $r-$tuple $(\sigma_{1},\cdots,\sigma_{r}) \in G^{r}$ is a \textit{generating vector}  for $G$ of type $(0;m_{1},\cdots,m_{r})$ if this satisfy the three preceding conditions replace $\theta(x_{j})$ by $\sigma_{j}$.


    \noindent Thus given the action on $S$ we have a generating vector of type the signature of this action. The reciprocal is know  as the Existence Riemann theorem. For more detail see \cite{broughton}.

\section{Families of Riemann surfaces with equivalent actions}

\noindent In this section we produce for each $n$,  two families of Riemann surfaces of genus  $3(2^{n}-1)$ with group of automorphisms of order $2^{n+2}$ and signature $(0;2^{n+1},2^{n+1},2^{n},2)$. For one of the families the group we will be abelian, and for the other it will be a semidirect product. In both cases, we will have that there exist two cyclic subgroups which define directly topologically, but not conformally, equivalent actions.

\begin{theo}\label{defi}
Let $f_{a,\lambda}$ be the polynomial given by
    \begin{equation*}
    f_{a,\lambda}\left(x,y\right)=y^{2^{n}}-x^{a}\left(x^{2}-1\right)^{a}\left(x^{2}-\lambda^{2}\right)\left(x^{2}-\lambda^{-2}\right)
    \end{equation*}
where $n,a \in \mathbb{N}, \,$ $ \lambda\in \mathbb{C}, \,$ and  $\,\lambda^{4}\neq 1,0$.

\noindent Then, for each odd number $a$ and each $\lambda$, we have that $f_{a,\lambda}$ defines a Riemann surface $S_{a,\lambda}$ of genus $3(2^{n}-1)$.

\noindent Furthermore, the possible singular points for the homogeneous polynomial associated to $f_{a,\lambda}$ are
    \[\begin{tabular}{c|c|c|c|c}
     case & $\sign(2^{n}-3a-4)$  & $a$ & condition   &  singular points\\
    \hline
    $1$ & $+$ &  $1$ &  $n=3$ & \small $\emptyset$ \\
    \hline
    $2$ & $+$ & $\neq 1$ &  $2^{n}-3a-5=0$ & \small $\{[0,0,1],[1,0,1],[-1,0,1]\}$ \\
     \hline
    $3$ & $+$ &  $1$ &  $2^{n}-3a-5\neq 0$ & \small $\{[1,0,0]\}$ \\
     \hline
     $4$ & $ +$ & $\neq 1$ &  $2^{n}-3a-5\neq 0$ & \small $\{[0,0,1],[1,0,0],[1,0,1],[-1,0,1]\}$ \\
     \hline
    $5$ & $-$ & $\neq 1$ & &   \small $\{[0,0,1],[0,1,0],[1,0,1],[-1,0,1]\}$ \\
     \hline
    $6$ & $-$ & $1$ &  $n=1,2$ & \small $\{[0,1,0]\}$ \\
     \hline
    \end{tabular}\]

    Table 1: Singular points

\end{theo}

\begin{proof}
\noindent Observe that if $2^{n}-3a-4=0$ then $a$ is an even number. Moreover  $2^{n}-5-3a=0$ for some odd number $a$, if only if $n$ is an odd number. Particularly if $a=1$ then $n=3$.

\noindent We have two cases for the homogeneous polynomial associated to $f_ {a,\lambda}$.
    \begin{enumerate}
     \item Case $2^{n}-3a-4 > 0$. \\
        \begin{equation*}
            F_{1}\left(X,Y,Z\right)=Y^{2^{n}}-X^{a}Z^{2^{n}-3a-4}\left(X^{2}-Z^{2}\right)^{a}\left(X^{2}-\lambda^{2}Z^{2}\right)\left(X^{2}-\lambda^{-2}Z^{2}\right)
        \end{equation*}
where $[X,Y,Z] \in \mathbb{P}^{2}\mathbb{C}$ (Projective plane) and $x=\frac{X}{Z}$ , $y =\frac{Y}{Z}$.

\noindent For
    \begin{itemize}
    \item $a \neq  1, 2^{n}-3a-5 \neq 0$ , we have $4$ singular points, they are
    $$\left\{[0,0,1],[1,0,0], [1,0,1],[-1,0,1]\right\}\,$$

    \noindent By the Normalization process we have the following charts for theses points
    \[\begin{tabular}{c|c|c}
   singular point & local coordinate & \\
   \hline \text{  }$[0,0,1]$ & $s\rightsquigarrow [s^{2^{n}},s^{a}h_{0}(s),1]$ & $h_{0}(0)\neq 0$ \\
   \hline \text{  }$[1,0,1]$ & $s\rightsquigarrow [s^{2^{n}}+1,s^{a}h_{1}(s),1]$ & $h_{1}(0)\neq 0$ \\
   \hline \text{  }$[-1,0,1]$ & $s\rightsquigarrow [s^{2^{n}}-1,s^{a}h_{-1}(s),1]$ & $h_{-1}(0)\neq 0$ \\
   \hline \text{  }$[1,0,0]$ & $t\rightsquigarrow [1,t^{2^{n}-3a-4}h_{\infty}(t),t^{2^{n}}]$ & $h_{\infty}(0)\neq 0$\\
   \hline
   \end{tabular}\]
 \textit{Table 2: Local coordinates for singular points case $F_{1}$}

 \bigskip
\noindent where for each $j$\, , $h_{j}$ is an holomorphic maps defined on an open subset of complex plane.

    \item $a \neq 1, 2^{n}-3a-5 = 0$, then $n$ is an odd number and the singular points are $\left\{[0,0,1], [1,0,1],[-1,0,1]\right\}\, .$
    \item $a = 1, 2^{n}-3a-5 \neq 0$, then $n>3$, since $2^{n}-7>0 \text{ and } 2^{n}-8\neq0\, .$
    \noindent Then the singular points are $\left\{[1,0,0]\right\}$
    \item $a = 1, 2^{n}-3a-5 = 0$, then $n=3$, since $2^{n}-7>0 \text{ and } 2^{n}-8=0$.
    \noindent Then in this case $F_{1}$ has not  singular points.
    \end{itemize}

    \item Case $2^{n}-3a-4 < 0$.
        \begin{equation*}
            F_{2}\left(X,Y,Z\right)=Y^{2^{n}}Z^{3a+4-2^{n}}-X^{a}\left(X^{2}-Z^{2}\right)^{a}\left(X^{2}-\lambda^{2}Z^{2}\right)\left(X^{2}-\lambda^{-2}Z^{2}\right)\,.
         \end{equation*}

\noindent If $a \neq  1 $ , we have that the singular points are
$$\left\{[0,0,1], [1,0,1],[-1,0,1], [0,1,0]\right\}$$

\noindent By the Normalization process we have the following charts for theses points
   \[\begin{tabular}{c|c|c}
   point & local coordinate & \\
   \hline \text{  }$[0,0,1]$ & $s\rightsquigarrow [s^{2^{n}},s^{a}g_{0}(s),1]$ & $g_{0}(0)= 1$ \\
   \hline \text{  }$[1,0,1]$ & $s\rightsquigarrow [s^{2^{n}}+1,s^{a}g_{1}(s),1]$ & $g_{1}(0)\neq 0$ \\
   \hline \text{  }$[-1,0,1]$ & $s\rightsquigarrow [s^{2^{n}}-1,s^{a}g_{-1}(s),1]$ & $g_{-1}(0)\neq 0$ \\
   \hline \text{  }$[0,1,0]$ & $t\rightsquigarrow [t^{3a+4-2^{n}}g_{\infty}(t),1,t^{3a+4}]$ & $g_{\infty}(0)=1$ \\
   \hline
   \end{tabular}\]
 \textit{Table 3: Local coordinates for singular points case $F_{2}$}

\bigskip
\noindent where for each $j$\, , $g_{j}$ is an holomorphic maps defined on an open subset of complex plane.

\noindent Now if $a=1$ as $2^{n}-3a-4<0$ then  $n\leq 2$. These cases are studied by G. Gonz\'alez-Diez and R. Hidalgo in \cite{gabrub}.
 \end{enumerate}

\noindent Now using the Normalization process  we get a Riemann Surfaces of genus $3(2^{n}-1)$. This is because we may define the holomorphic map
$\pi: S_{a,\lambda} \longrightarrow \widehat{\mathbb{C}}$ given by $\pi(x,y) = x$.

\noindent Note that $\pi$ has degree $2^n$. The ramification point set of  $\pi$ is
\begin{equation}\label{ramipoints}
B= \{P_{0}, Q_{0},[\pm 1,0,1],[\pm \lambda,0,1], [\pm \lambda^{-1},0,1]\}
\end{equation}
where $P_{0}=[0,0,1]$ and either $Q_{0}= [1,0,0]$ (if $2^{n}-3a-4>0$) or $Q_{0}=[0,1,0]$ (if $2^{n}-3a-4<0$). The branch point set is $\pi(B)=\{0, \infty, \pm \lambda, \pm \lambda^{-1}, \pm 1\}$. For each $P\in B$ we have the multiplicity of $P$ is $2^{n}$.

\noindent By the Riemann--Hurwitz formula it follow the genus of $S$ is $3 \left(2^{n}-1\right)$.
\end{proof}

\bigskip
\noindent The following theorem yields information about the automorphisms group of $S_{a,\lambda}$. The interest of the  theorem is in the assertion that for each Riemann surface $S_{a,\lambda}$ the automorphisms group is not trivial (except some cases).

    \begin{theo}\label{auto}
    Let $\tau_{1},\tau_{2}$ be the self--maps of $S_{a,\lambda}$  defined by
        \begin{eqnarray*}
        \tau_{1}(x,y)&=&(-x,\omega_{2^{n+1}}y) \\
        \tau_{2}(x,y)&=& \left(\dfrac{1}{x},\dfrac{\omega_{2^{n+1}}y}{x^{c}}\right)
        \end{eqnarray*}
    where $c$ is a natural number  determined by the  equation $c\cdot 2^{n-1} = 2a+2$.

    \noindent Then $\tau_{1},\tau_{2} \in \auto(S_{a,\lambda})$ and they have order $2^{n+1}$ each.

    \noindent Furthermore depending on the values for $a$ we have:
        \begin{enumerate}
        \item If $c$ is an even number, then $G_{1}=<\tau_{1},\tau_{2}>$ is an abelian group, isomorphic to $\mathbb{Z}/2^{n+1}\mathbb{Z} \times \mathbb{Z}/2\mathbb{Z}$.\\
            We call $\mathfrak{S}_{1}$ the corresponding family of surfaces.
        \item If $c$ is an odd number, then $G_{2}=<\tau_{1},\tau_{2}>$ is a group isomorphic to $\mathbb{Z}/2^{n+1}\mathbb{Z} \rtimes_{h} \mathbb{Z}/2\mathbb{Z}$, where $\tau_{2}\tau_{1}=\tau_{1}^{2^{n}+1}\tau_{2}$.\\
            We call $\mathfrak{S}_{2}$ the corresponding family of surfaces.
        \end{enumerate}
    \end{theo}

\begin{proof}
It is not difficult to check that $f_{a,\lambda}(\tau_{1}(x,y))=0$ and $f_{a,\lambda}(\tau_{2}(x,y))=0$.

\noindent Note that
$$\tau_{1}^{2}(x,y)=\tau_{2}^{2}(x,y)=(x,\omega_{2^{n+1}}^{2}y)$$
Then the order of $\tau_{1}^{2}$ is $2^{n}$, since $\omega_{2^{n+1}}^{2}$ is a $2^{n}$-th primitive root of unity, therefore  $\tau_{1}$ and $\tau_{2}$ has order $2^{n+1}$.

\noindent Now we consider the homogeneous coordinates for $\tau_{1},\tau_{2}$. They are given by
        \begin{eqnarray*}
        \tau_{1}[X,Y,Z]&=&[-X,\omega_{2^{n+1}} Y,Z] \\
        \tau_{2}[X,Y,Z]&=&[ZX^{c-1},\omega_{2^{n+1}} YZ^{c-1},X^{c}]\,
        \end{eqnarray*}

\noindent To prove that $\tau_{1},\tau_{2}\in \auto(S_{a,\lambda})$, we must  prove that  for any charts $(U,\varphi),(V,\phi)$, on $S_{a,\lambda}$, such that $\tau_{j}(V) \cap U \neq \emptyset$ , the map $\varphi \circ \tau_{j} \circ \phi^{-1}\, ,$ is a holomorphic function (on  some subset of $\mathbb{C}$).

\noindent We will do the computations for $\tau_{2}$ at $P_{0}$ and suppose $2^{n}-3a-4 > 0$.

\noindent Using the chart at $P_{0}$ (see table 2, theorem \ref{defi})  we have
        \begin{eqnarray*}
        \tau_{2}[s^{2^{n}},s^{a}h_{0}(s),1]&=& \left[1,s^{2^{n}-3a-4}\omega_{2^{n+1}}h_{0}(s),s^{2^{n}}\right]\, .
        \end{eqnarray*}
then  $\tau_{2}(P_{0})=Q_{0}$.

    \noindent Now consider the  chart at $Q_{0}$  given in the table 2, theorem \ref{defi}. Then
        \begin{eqnarray*}
         \footnotesize \varphi \circ \tau_{2} \circ \phi^{-1}(s)= \varphi \circ \tau_{2}[s^{2^{n}},s^{a}h_{0}(s),1] = \varphi [1,s^{2^{n}-3a-4}\omega_{2^{n+1}}h_{0}(s),s^{2^{n}}]
         =\omega s
        \end{eqnarray*}
        where $\omega$ is a $2^{n}$-th root of unity (we recall that $t^{2^{n}}=s^{2^{n}}$). Therefore, the map is a holomorphic function.

    \noindent It is not difficult to verify the preceding process for the other charts.

    \noindent Recalling that $2^{n-1}c=2a+2$,  we have
        \begin{eqnarray*}
        \tau_{1}\tau_{2}[X,Y,Z]&=&[-ZX^{c-1},\omega_{2^{n+1}}^{2} Z^{c-1}Y,X^{c}]\\
        \\
        \tau_{2}\tau_{1}[X,Y,Z]
                            &=&\left\{\begin{array}{cc}
                            \tau_{1}\tau_{2}[X,Y,Z] &, c \text{ is an even number}\\
                            \\
                            \tau_{1}^{2^{n}+1}\tau_{2}[X,Y,Z] &, c \text{ is an odd number }\\
                            \end{array} \right.\, .
        \end{eqnarray*}




\noindent It is not difficult compute the elements of the group generated by $\tau_{1}$ and $\tau_{2}$. The cardinality of this group is $2^{n+2}$.

\noindent If $c$ is an even number then the group is abelian. It is not difficult prove that the group  $<\tau_{1},\tau_{2}>$ is isomorphic to $\mathbb{Z}/2^{n+1}\mathbb{Z}\times \mathbb{Z}/2\mathbb{Z}$. In this case we call $G_{1}=<\tau_{1},\tau_{2}>$.

\noindent Now if $c$ is an odd number then the group $<\tau_{1},\tau_{2}>$ is isomorphic to $\mathbb{Z}/ 2^{n+1} \mathbb{Z}\rtimes_{h} \mathbb{Z}/ 2 \mathbb{Z}$. In fact,  the element $\nu=\tau_{1}^{2^{n-1}-1}\tau_{2}$ has order $2$ and we may define $h:<\nu> \rightarrow \auto(<\tau_{1}>)$ given by  $h(\nu)(\tau_{1})=\tau_{1}^{2n+1}$.  In this case we call $G_{2}=<\tau_{1},\tau_{2}>$.
\end{proof}

\noindent We remark that for $a=1$ as $c\in \mathbb{Z}$ then   $ n\leq 3$. If $n<3$ then $c$ is an even number and theses cases were studied \cite{gabrub}. If $n=3$ then $c=1$. Furthermore we conclude  that the case (3) in the table 1, theorem \ref{defi}, it has not automorphisms of type $\tau_{j}$ for $j=1,2$.

\noindent For the case (2) in the table 1, theorem \ref{defi}, also it has not automorphisms of type $\tau_{j}$ for $j=1,2$.

\bigskip

\noindent By the preceding theorem for $S_{1}\in \mathfrak{S}_{1}$ we have that the group $\mathbb{Z}/2^{n+1}\mathbb{Z}\times \mathbb{Z}/2\mathbb{Z}\simeq G_{1}$ acts on $S_{1}$. Now we are interest in compute the signature of this action and the signature of yours subgroups. The following theorem summarize this information.

\begin{theo}\label{accionc}
    The  cyclic subgroups of $G_{1}$ acting with fixed points (different) are given as follows:
        \begin{enumerate}
        \item $H_{1}=<\tau_{1}>$ subgroup of order $2^{n+1}$, acting on $S_{1}$ with signature \\ $(0;2^{n+1},2^{n+1},2^{n},2^{n},2^{n})$.
        \item $H_{2}=<\tau_{2}>$ subgroup of order $2^{n+1}$, acting on $S_{1}$ with signature \\ $(0;2^{n+1},2^{n+1},2^{n},2^{n},2^{n})$.
        \item $H_{3}=<\tau_{1}^{2}>$ subgroup of order $2^{n}$, acting on $S_{1}$ with signature \\ $(0;2^{n},2^{n},2^{n},2^{n},2^{n},2^{n},2^{n},2^{n})$.
        \item $H_{4}=<\tau_{1}^{2^{n-1}c-1}\tau_{2}>$ subgroup of order $2$, acting on $S_{1}$ with signature $(2^{n}-1;2^{n+1}\cdot 2)$
        ($\tau_{1}^{2^{n-1}c-1}\tau_{2}$ has $2^{n+1}$ fixed points).
        \end{enumerate}
    Furthermore, we have that the group $G_{1}$ acts on $S_{1}$ with signature $(0;2^{n+1},2^{n+1},2^{n},2)$ and that
    $G_{1}=\auto(S_{1})$, except for finitely many $S_{1}\in \mathfrak{S}_{1}$.
    \end{theo}

\begin{proof}
    \noindent We will compute the fixed points.

\medskip
    \noindent For  $\tau_{1}$ the fixed points on $S_{1}$ are $\{P_{0},Q_{0}\}$  (recall equation \eqref{ramipoints}).

\medskip
     \noindent For $\tau_{2}$ the fixed points on $S_{1}$ are $\left\{[1,0,1],[-1,0,1]\right\}$. Remark that  $\tau_{2}(P_{0})=Q_{0}$.

\medskip
    \noindent For $\tau_{1}^{j}$ where $j$ is an even number the set of fixed points on $S_{1}$ is $B$ (recall equation \eqref{ramipoints}).

\medskip
    \noindent If $\tau^{j}\tau_{2}$ has fixed point then $j= 3 \cdot 2^{n-1}c-1$ or $j=2^{n-1}c-1$. Since $c$ is an even number  then
    \begin{equation} \label{fixed}
    \tau_{1}^{3 \cdot 2^{n-1}c-1} \tau_{2} = \tau_{1}^{2^{n-1}c-1} \tau_{2}
    \end{equation}
    this is because $3 \cdot 2^{n-1}c-1 \equiv 2^{n-1}c-1 \mod{2^{n+1}}$. In this case the fixed points on $S_{1}$ are
     \begin{equation*}
     \left\{ \begin{array}{cc}
     \text{ } [i,p,1] :& p^{2^{n}}=-(2i)^{a}(1+\lambda^{2})(1+\lambda^{-2}) \\
     \text{ } [-i,q,1] :&  q^{2^{n}}=(2i)^{a}(1+\lambda^{2})(1+\lambda^{-2})
     \end{array}\right\}
     \end{equation*}

\noindent It is not difficult to compute for each $H_{j}$ the signature using  Riemann Hurwitz formula.

\noindent Let $\Gamma$ be a Fuchsian group with the above signature. By Singerman \cite{singerman}, there is no
other Fuchsian group with signature of the form $(0;a,b,c,d)$ that contains it strictly. It
follows that it may only be contained in a triangular signature. Hence by dimension
arguments,  $\Gamma$ cannot be contained strictly in other subgroup as finite index
subgroup except for a finite number of possibilities (up to conjugation by M\"obius
transformations). Therefore, the family of Riemann surfaces does not have any other  automorphisms than those of $G_{1}$, except for finitely many $S_{1} \in \mathfrak{S}_{1}$.

    \end{proof}

\noindent We can now state the analogue of the preceding theorem for $S_{2}\in \mathfrak{S}_{2}$.
    \begin{theo}\label{accionn}
    The  cyclic subgroups of $G_{2}$ acting with fixed points (different) are given as follows:
        \begin{enumerate}
        \item $H_{1}=<\tau_{1}>$ subgroup of order $2^{n+1}$, acting on $S_{2}$ with signature \\ $(0;2^{n+1},2^{n+1},2^{n},2^{n},2^{n})$.
        \item $H_{2}=<\tau_{2}>$ subgroup of order $2^{n+1}$, acting on $S_{2}$ with signature \\ $(0;2^{n+1},2^{n+1},2^{n},2^{n},2^{n})$.
        \item $H_{3}=<\tau_{1}^{2}>$ subgroup of order $2^{n}$, acting on $S_{2}$ with signature \\ $(0;2^{n},2^{n},2^{n},2^{n},2^{n},2^{n},2^{n},2^{n})$.
        \item $H_{4}=<\tau_{1}^{3\cdot2^{n-1}c-1}\tau_{2}>$ subgroup of order $2$, acting on $S_{2}$ with signature
        $(5\cdot 2^{n-2}-1;2^{n}\cdot 2)$
        ($\tau_{1}^{3\cdot 2^{n-1}c-1}\tau_{2}$ has $2^{n}$ fixed points).
        \item $H_{5}=<\tau_{1}^{2^{n-1}c-1}\tau_{2}>$ subgroup of order $2$, acting on $S_{2}$ with signature $(5\cdot2^{n-2}-1;2^{n}\cdot 2)$
        ($\tau_{1}^{2^{n-1}c-1}\tau_{2}$ has $2^{n}$ fixed points).\\
        $H_{5}$ is a subgroup conjugate to $H_{4}$.
        \end{enumerate}
    Furthermore, we have that the group $G_{2}$ acts on $S_{2}$ with signature $(0;2^{n+1},2^{n+1},2^{n},2)$ and that
    $G_{2}=\auto(S_{2})$, except for finitely many $S_{2}\in \mathfrak{S}_{2}$.
    \end{theo}

\noindent The proof is similar to that  theorem  \ref{accionc}. Recall that in this case  $c$ is an odd number. Then the equation \eqref{fixed} is not true. Therefore there exist two subgroups  of order $2$ that it has fixed points. It is not difficult prove that these groups are conjugate. For any these case it has $2^{n}$ fixed points and the same signature.

\bigskip

\noindent For all $P$  fixed point given by the theorems \ref{accionc} and \ref{accionn} we may compute the value of $\delta_{P}$.
\noindent According the remarks of  the theorem \ref{auto} we have just $3$ cases to compute.  In the table 1 ( theorem \ref{defi}) they are: case (1), case (4) and case (5). Remark that the case (6) was study in \cite{gabrub}.

\bigskip

    \noindent Case (1).  $n=3$\, , $a=1$\, , $c=1$
    \[\begin{tabular}{c|c|c|c}
    order & $\tau$ & $P$ & $\delta_{P}(\tau)$ \\
    \hline $16$ & $\tau_{1}$ & $[0,0,1]$ & $\omega_{16}$ \\
    \hline $16$ & $\tau_{1}$ & $[1,0,0]$ & $\omega_{16}^{9}$ \\
    \hline $16$ & $\tau_{1}^{9}$ & $[1,0,0]$ & $\omega_{16}$ \\
    \hline $16$ & $\tau_{2}$ & $[1,0,1]$ & $\omega_{16}$ \\
    \hline $16$ & $\tau_{2}$ & $[-1,0,1]$ & $\omega_{16}^{9}$ \\
    \hline $16$ & $\tau_{2}^{9}$ & $[-1,0,1]$ & $\omega_{16}$\\
    \hline $8$ & $\tau_{1}^{2}$ &  $[\lambda,0,1]$ & $\omega_{8}$\\
    \hline $8$ & $\tau_{1}^{2}$ &  $[-\lambda,0,1]$ & $\omega_{8}$ \\
    \hline $8$ & $\tau_{1}^{2}$ &  $[\lambda^{-1},0,1]$  & $\omega_{8}$ \\
    \hline $8$ & $\tau_{1}^{2}$ &  $[-\lambda^{-1},0,1]$  & $\omega_{8}$ \\
    \hline
    \end{tabular}\]

\bigskip
    \noindent Case (4). $2^{n}-3a-4>0\, ,$ $a> 1\, ,$  $2^{n}-3a-5\neq 0$
    \[\begin{tabular}{c|c|c|c|c}
    order &  $\tau$  & $P$ &   local auto. &  $\delta_{P}(\tau)$ \\
    \hline $2^{n+1}$ & $\tau_{1}$ &  $[0,0,1]$ &
    { \footnotesize \begin{tabular}{c}
    $s \rightsquigarrow \omega_{2^{n+1}}^{j_{0}}s$ \\
    $j_{0}a \equiv 1 \mod 2^{n+1}$ \\
    $j_{0}$ is an odd number
    \end{tabular}}
    & $\omega_{2^{n+1}}^{j_{0}}$\\
    \hline $2^{n+1}$ & $\tau_{1}^{a}$ &  $[0,0,1]$ & & $\omega_{2^{n+1}}$ \\
    \hline $2^{n+1}$ & $\tau_{1}$ &  $[1,0,0]$ &
     {\footnotesize\begin{tabular}{c}
    $t \rightsquigarrow \omega_{2^{n+1}}^{j_{\infty}}t$ \\
    $j_{\infty}(-3a-4) \equiv 1 \mod 2^{n+1}$ \\
    $j_{\infty}$ is an odd number
    \end{tabular}}
    & $\omega_{2^{n+1}}^{j_{\infty}}$\\
    \hline $2^{n+1}$ & $\tau_{1}^{-3a-4}$ &  $[1,0,0]$ & & $\omega_{2^{n+1}}$ \\
    \hline $2^{n+1}$ & $\tau_{2}$ &  $[1,0,1]$ &
     {\footnotesize\begin{tabular}{c}
    $s \rightsquigarrow \omega_{2^{n+1}}^{j_{1}}s$ \\
    $j_{1}a \equiv 1 \mod 2^{n+1}$ \\
    $j_{1}$ is an odd number
    \end{tabular}}
    & $\omega_{2^{n+1}}^{j_{1}}$\\
    \hline $2^{n+1}$ & $\tau_{2}^{a}$ &  $[1,0,1]$ & & $\omega_{2^{n+1}}$ \\
    \hline $2^{n+1}$ & $\tau_{2}$ &  $[-1,0,1]$ &
    {\footnotesize \begin{tabular}{c}
    $s \rightsquigarrow \omega_{2^{n+1}}^{j_{-1}}s$ \\
    $j_{-1}a \equiv  1 \mod 2^{n+1}$ \\
    $j_{-1}$ is an odd number
    \end{tabular}}
     & \begin{tabular}{c}
     $\omega_{2^{n+1}}^{j_{-1}}$\\
     {\footnotesize $c$ is an even number}
    \end{tabular}\\
    \hline
    \end{tabular}\]

    \[\begin{tabular}{c|c|c|c|c}
    order &  $\tau$  & $P$ &   local auto. &  $\delta_{P}(\tau)$ \\
    \hline $2^{n+1}$ & $\tau_{2}^{a}$ &  $[-1,0,1]$ & & \begin{tabular}{c}
     $\omega_{2^{n+1}}$\\
     {\footnotesize $c$ is an even number}
    \end{tabular}\\
    \hline $2^{n+1}$ & $\tau_{2}$ &  $[-1,0,1]$ &
     {\footnotesize\begin{tabular}{c}
    $s \rightsquigarrow \omega_{2^{n+1}}^{j_{-1}}s$ \\
    $j_{-1}a \equiv  2^{n}+1 \mod 2^{n+1}$  \\
    $j_{-1}$ is an odd number
    \end{tabular}}
    & \begin{tabular}{c}
    $\omega_{2^{n+1}}^{j_{-1}}$\\
    {\footnotesize $c$ is an odd number}
    \end{tabular}\\
    \hline $2^{n+1}$ & $\tau_{2}^{a}$ &  $[-1,0,1]$ & & \begin{tabular}{c}
     $\omega_{2^{n+1}}^{2^{n}+1}$\\
     {\footnotesize $c$ is an odd number}
    \end{tabular}\\
    \hline $2^{n+1}$ & $\tau_{2}^{a(2^{n}+1)}$ &  $[-1,0,1]$ & & \begin{tabular}{c}
     $\omega_{2^{n+1}}$\\
     {\footnotesize $c$ is an odd number}
    \end{tabular}\\
    \hline $2^{n}$ & $\tau_{1}^{2}$ &  $[\lambda,0,1]$ &
    $y \rightsquigarrow \omega_{2^{n+1}}^{2}y$ & $\omega_{2^{n}}$\\
    \hline $2^{n}$ & $\tau_{1}^{2}$ &  $[-\lambda,0,1]$ &
    $y \rightsquigarrow \omega_{2^{n+1}}^{2}y$ & $\omega_{2^{n}}$ \\
    \hline $2^{n}$ & $\tau_{1}^{2}$ &  $[\lambda^{-1},0,1]$ &
    $y \rightsquigarrow \omega_{2^{n+1}}^{2}y$ & $\omega_{2^{n}}$ \\
    \hline $2^{n}$ & $\tau_{1}^{2}$ &  $[-\lambda^{-1},0,1]$ &
    $y \rightsquigarrow \omega_{2^{n+1}}^{2}y$ & $\omega_{2^{n}}$ \\
    \hline
    \end{tabular}\]

    \bigskip
\noindent Case (5). $2^{n}-3a-4<0\, ,$ $a> 1$
    \[\begin{tabular}{c|c|c|c}
    order &  $\tau$  & $P$ &  $\delta_{P}(\tau)$ \\
    \hline $2^{n+1}$ & $\tau_{1}$ &  $[0,0,1]$
    & $\omega_{2^{n+1}}^{j_{0}}$\\
    \hline $2^{n+1}$ & $\tau_{1}^{a}$ &  $[0,0,1]$ & $\omega_{2^{n+1}}$ \\
    \hline $2^{n+1}$ & $\tau_{1}$ &  $[0,1,0]$
    & $\omega_{2^{n+1}}^{j_{\infty}}$\\
    \hline $2^{n+1}$ & $\tau_{1}^{-3a-4}$ &  $[0,1,0]$  & $\omega_{2^{n+1}}$ \\
    \hline $2^{n+1}$ & $\tau_{2}$ &  $[1,0,1]$
    & $\omega_{2^{n+1}}^{j_{1}}$\\
    \hline $2^{n+1}$ & $\tau_{2}^{a}$ &  $[1,0,1]$  & $\omega_{2^{n+1}}$ \\
    \hline $2^{n+1}$ & $\tau_{2}$ &  $[-1,0,1]$
     & \begin{tabular}{c}
     $\omega_{2_{n+1}}^{j_{-1}}$\\
     {\footnotesize $c$ is an even number}
    \end{tabular}\\
    \hline $2^{n+1}$ & $\tau_{2}^{a}$ &  $[-1,0,1]$  & \begin{tabular}{c}
     $\omega_{2_{n+1}}$\\
     {\footnotesize $c$ is an even number}
    \end{tabular}\\
    \hline $2^{n+1}$ & $\tau_{2}$ &  $[-1,0,1]$
    & \begin{tabular}{c}
    $\omega_{2^{n+1}}^{j_{-1}}$\\
    {\footnotesize $c$ is an odd number}
    \end{tabular}\\
    \hline $2^{n+1}$ & $\tau_{2}^{a}$ &  $[-1,0,1]$  & \begin{tabular}{c}
     $\omega_{2_{n+1}}^{2^{n}+1}$\\
     {\footnotesize $c$ is an odd number}
    \end{tabular}\\
    \hline $2^{n+1}$ & $\tau_{2}^{a(2^{n}+1)}$ &  $[-1,0,1]$  & \begin{tabular}{c}
     $\omega_{2_{n+1}}$\\
     {\footnotesize $c$ is an odd number}
    \end{tabular}\\
    \hline $2^{n}$ & $\tau_{1}^{2}$ &  $[\lambda,0,1]$  & $\omega_{2^{n}}$\\
    \hline $2^{n}$ & $\tau_{1}^{2}$ &  $[-\lambda,0,1]$  & $\omega_{2^{n}}$ \\
    \hline $2^{n}$ & $\tau_{1}^{2}$ &  $[\lambda^{-1},0,1]$  & $\omega_{2^{n}}$ \\
    \hline $2^{n}$ & $\tau_{1}^{2}$ &  $[-\lambda^{-1},0,1]$  & $\omega_{2^{n}}$ \\
    \hline
    \end{tabular}\]

 \subsection{Geometric presentation}
\noindent  We have showed that, for each $S_{1}\in \mathfrak{S}_{1}$, the group
    $$G_{1}\simeq \mathbb{Z}/2^{n+1}\mathbb{Z} \times \mathbb{Z}/2\mathbb{Z}\, $$
acts on $S_{1}$ with signature $(0;2^{n+1},2^{n+1},2^{n},2)$.
\noindent For this action, we have that the 4--tuple
    $$(\tau_{1}^{a},\tau_{2}^{a},\tau_{1}^{2},\tau_{1}^{2^{n-1}c-1}\tau_{2})\, $$
is a generating vector. Using the polygon method \cite{brr}, according to signature,
we find a presentation for $G_{1}$. We thus obtain the polygon
\[\begin{tabular}{c}
 \includegraphics[height=5 cm]{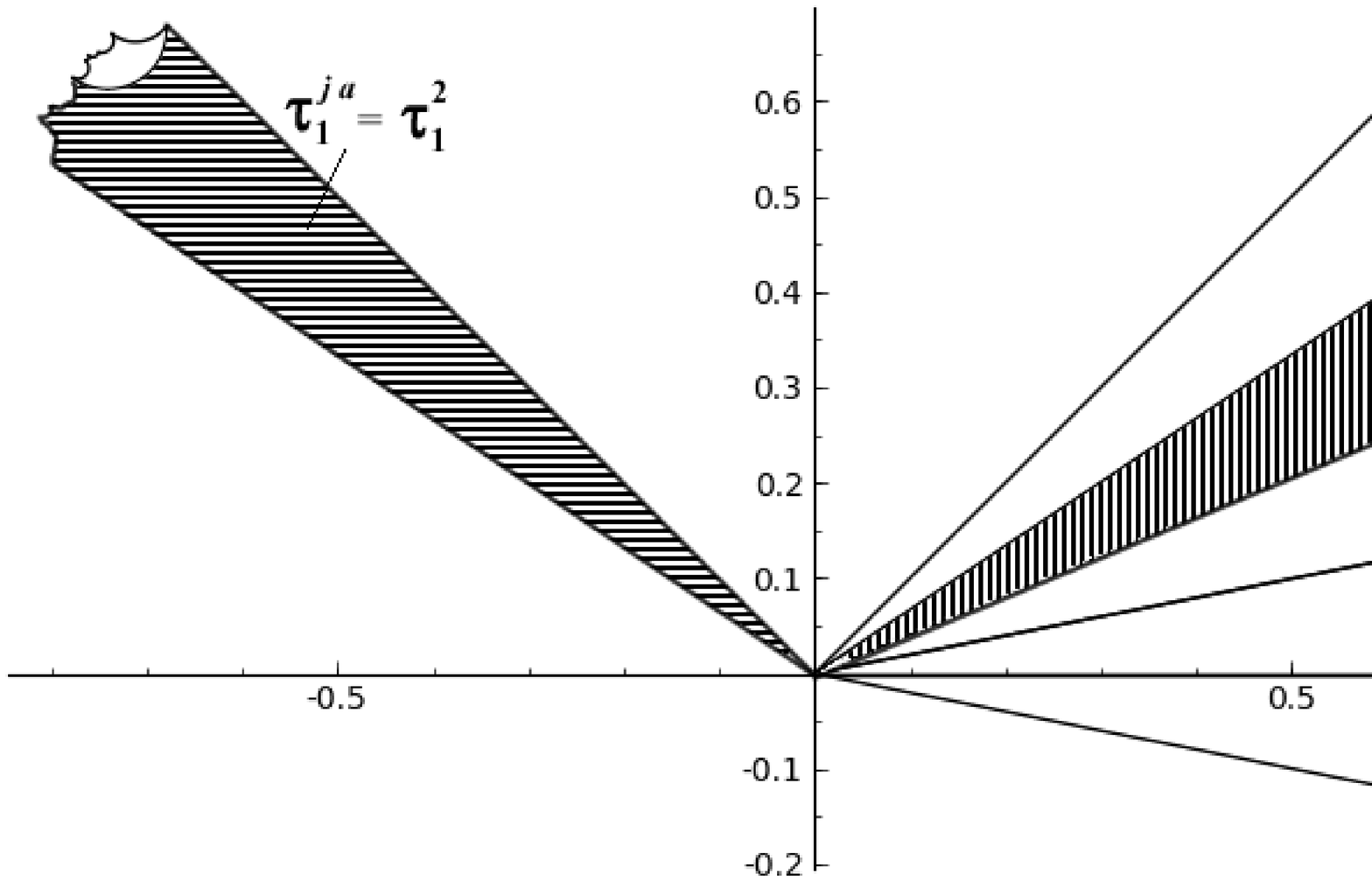}
\end{tabular}\]

\noindent We call $D_{1}\,, D_{2}\,,D_{3}$ the elements associated respectively to $\tau_{1}^{a},\, \tau_{2}^{a},\, \tau_{1}^{2}$. From the picture we conclude the relationships
\begin{eqnarray*}
 D_{3}D_{1}^{4}D_{3}&=&1, \qquad \text{grid} \\
 D_{2}D_{1}D_{2}^{-1}D_{1}^{-1}&=&1, \qquad \text{dots} \\
 D_{1}^{2}D_{2}^{-2}&=&1, \qquad \text{vertical lines} \\
 D_{1}^{j}D_{3}^{-1}&=&1, \qquad \text{horizontal lines}
\end{eqnarray*}

\noindent The reader should keep in mind that we are in the case where  $c$ is an even number and  $2^{n-1}c=2a+2$.
It is not difficult to  verify that $j=2^{n-1}c-2$.

\noindent The following proposition is a consequence.

\begin{prop}
$G_{1}$ has a presentation of the form
\begin{equation} \label{group2n}
\left<D_{1},D_{2},D_{3},D_{4}:\begin{array}{c}
D_{1}^{2^{n+1}} = D_{2}^{2^{n+1}}=D_{3}^{2^{n}}=D_{4}^{2}=1,\\
 D_{1}D_{2}D_{3}D_{4}=1,\quad  D_{1}^{2}D_{2}^{-2}=1, \quad D_{1}^{c2^{n-1}-2}D_{3}^{-1}=1
\end{array}
\right>\, ,
\end{equation}
where $c$ is an even number.
\end{prop}

\begin{proof}
\noindent Call $\widetilde{G}_{1}$ to a group with presentation \eqref{group2n}.

\noindent First we may see that $D_{1}D_{2}=D_{2}D_{1}$. This is because $D_{1}^{2}=D_{2}^{2}$ and $D_{3}=D_{1}^{c2^{n-1}-2}$ then $D_{4}=D_{1}^{c2^{n-1}-1}D_{2}$ and it has order $2$.

\noindent Now we have
    $$\widetilde{G}_{1}=\{D_{1}^{j}D_{2}^{i}: 0\leq j<2^{n+1}, 0 \leq i\leq 1\}$$

\noindent  Then $\widetilde{G}_{1}$ is an abelian group of order $2^{n+2}$. Furthermore $\widetilde{G}_{1}$ is isomorphic to $G_{1}$. This is because we may define an isomorphism $\Phi$ given by $\Phi(D_{1})=\tau_{1}$ and $\Phi(D_{2})=\tau_{2}$.
\end{proof}

\bigskip

\noindent Also we have showed that, for each $S_{2}\in \mathfrak{S}_{2}$, the group
    $$G_{2}\simeq\mathbb{Z}/2^{n+1}\mathbb{Z} \rtimes_{h} \mathbb{Z}/2\mathbb{Z}\, $$
acts on $S_{2}$ with signature $(0;2^{n+1},2^{n+1},2^{n},2)$.
\noindent For this action, we have that the 4--tuple
    $$(\tau_{1}^{a},\tau_{2}^{a},\tau_{1}^{2},\tau_{1}^{2^{n-1}c-1}\tau_{2})\, $$
is a generating vector. As before, using the polygon method \cite{brr}, according to signature,
we find a presentation for $G_{2}$. We thus obtain the polygon

\[\begin{tabular}{c}
 \includegraphics[height=5 cm]{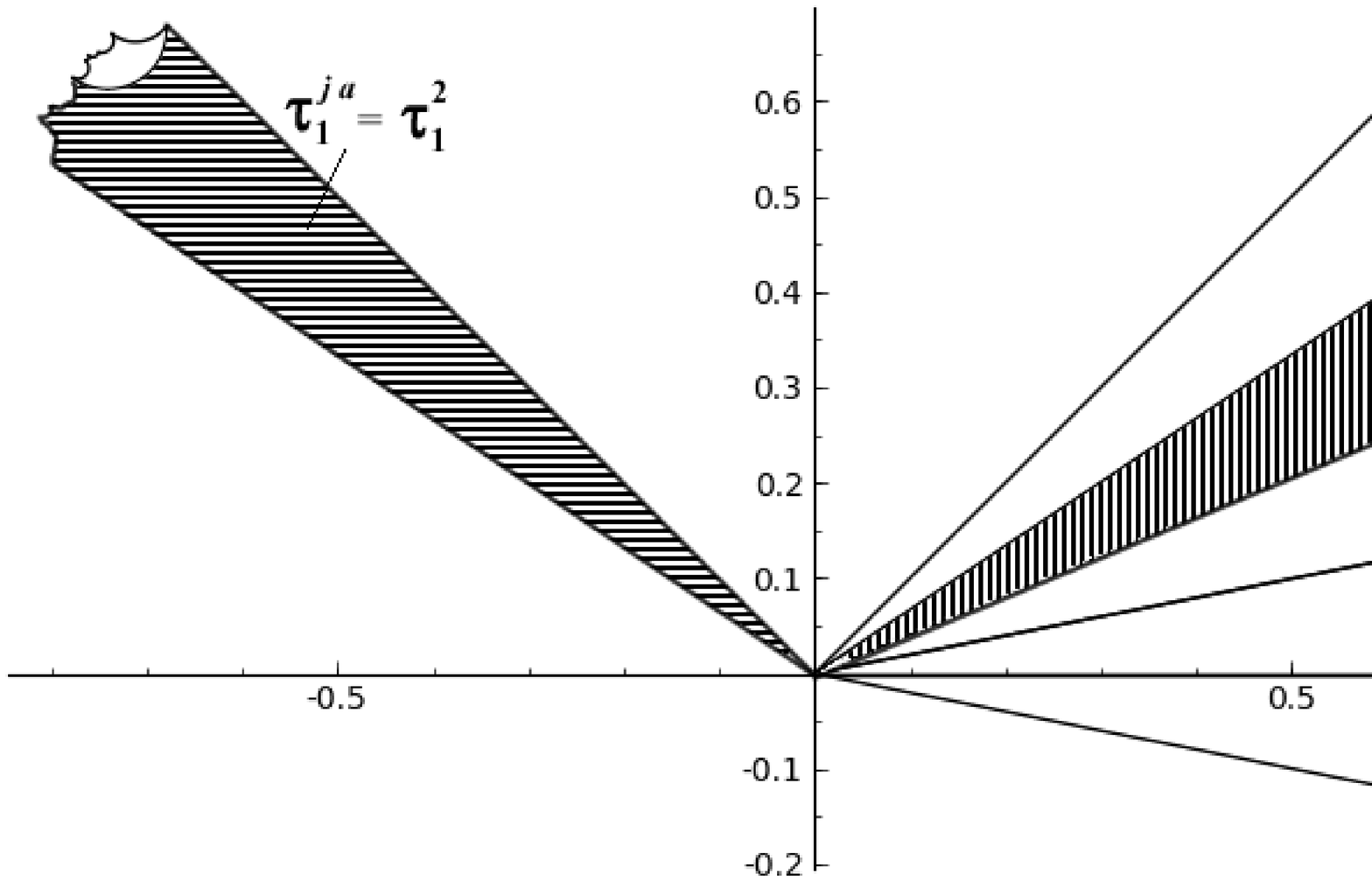}
\end{tabular}\]

\noindent We call $D_{1}\,, D_{2}\,,D_{3}$ the elements associated respectively to $\tau_{1}^{a},\, \tau_{2}^{a},\, \tau_{1}^{2}$. From the picture we conclude the relationships
\begin{eqnarray*}
 D_{1}^{2}D_{2}^{-2}&=&1, \qquad \text{vertical lines} \\
 D_{1}^{j}D_{3}^{-1}&=&1, \qquad \text{horizontal lines}
\end{eqnarray*}

\noindent The reader should keep in mind that we are in the case where  $c$ is an odd number and  $2^{n-1}c=2a+2$.
It is not difficult to  verify that $j=3\cdot 2^{n-1}c-2$, for $n>3$. For $n=3$ we have $a=1$ hence $j=2$.

\noindent The following proposition is a consequence.

\begin{prop}
For $n>3$, $G_{2}$ has a presentation of the form
\begin{equation}\label{exnon}
\left<D_{1},D_{2},D_{3},D_{4}:\begin{array}{c}
D_{1}^{2^{n+1}} = D_{2}^{2^{n+1}}=D_{3}^{2^{n}}=D_{4}^{2}=1,\\
D_{1}D_{2}D_{3}D_{4}=1,\quad  D_{1}^{2}D_{2}^{-2}=1, \quad D_{1}^{3\cdot 2^{n-1}c-2}D_{3}^{-1}=1
\end{array}
\right>\, .
\end{equation}
\end{prop}

\begin{proof}
\noindent Call $\widetilde{G}_{2}$ to a group with presentation \eqref{exnon}.

\noindent First we will see that $D_{2}D_{1}=D_{1}^{2^{n}+1}D_{2}$. In fact as $D_{4}$ has order $2$ we have  $D_{4}=D_{1}D_{2}D_{3}$ then
     $$ D_{2}D_{1}=D_{1}^{3-3\cdot 2^{n}c}D_{2}^{-1}=D_{1}^{1-3\cdot 2^{n}c}D_{2}$$
but as $c$ is an odd number we have $1-3\cdot 2^{n}c \equiv 2^{n}+1 \, \mod 2^{n+1}\, .$

\noindent Therefore we may define an isomorphism $\Phi: \widetilde{G}_{2}  \rightarrow G_{2} $ given by $\Phi(D_{1})=\tau_{1}$ and $\Phi(D_{2})=\tau_{2}$.
\end{proof}

\subsection{Classification of actions on $\mathfrak{S}_{1}$ and $\mathfrak{S}_{2}$}

\noindent In this section we will study the epimorphisms associated to the actions defined in the previous section.
 From now on we consider the notations as in theorem \ref{existence2}.
Recall that in any case (abelian and no abelian) $H_{1}=<\tau_{1}>$ is a subgroup of $\auto(S_{i})$ for $i=1,2$, and acts with signature
$(0;2^{n+1},2^{n+1},2^{n},2^{n},2^{n})$. Then there exist a Fuchsian group $\Gamma_{1,i}$ with signature  $(0;2^{n+1},2^{n+1},2^{n},2^{n},2^{n})$ and an epimorphism $\theta_{1,i}:\Gamma_{1,i} \rightarrow H_{1}$. We consider for $\Gamma_{1,i}$ a presentation according to your  hyperbolic polygon associated
$$\Gamma_{1,i}=\left<x_{1,i},\cdots,x_{5,i}: x_{1,i}^{2^{n+1}}=x_{2,i}^{2^{n+1}}=x_{3,i}^{2^{n}}=x_{4,i}^{2^{n}}=x_{5,i}^{2^{n}}=1= x_{1,i}x_{2,i}x_{3,i}x_{4,i}x_{5,i}\right>$$
where for each $j=1,2,3,4$ we have  $x_{j,i}$ is a positive minimal rotation.

\noindent Moreover there exists an isomorphisms $f_{1,i}:\Delta/K_{1,i} \rightarrow S_{i}$, where $K_{1,i}=\ker(\theta_{1,i})$. Call $P_{j,i}=f_{1,i}\left(\pi_{K_{1,i}}(z_{j,i})\right)$ then we have
\begin{eqnarray*}
    \delta_{P_{1,i}}(\theta_{1,i}(x_{1,i}))=\omega_{2^{n+1}}, &
    \delta_{P_{2,i}}(\theta_{1,i}(x_{2,i}))=\omega_{2^{n+1}}, \\
    \delta_{P_{3,i}}(\theta_{1,i}(x_{3,i}))=\omega_{2^{n}}, &
    \delta_{P_{4,i}}(\theta_{1,i}(x_{4,i}))=\omega_{2^{n}}
\end{eqnarray*}

\noindent By the theorem \ref{harvey} and the tables with the computes of $\delta_{P}$ (case (1), case (4) , case (5)) we have that if
$$P_{1,i}=[0,0,1], \quad  P_{2,i}=\tau_{2}(P_{1,i}), \quad  P_{3,i}=[1,0,1], \quad P_{4,i}=[\lambda,0,1]\, ,$$
then
    $$\left(\theta_{1,i}(x_{1,i}), \theta_{1,i}(x_{2,i}), \theta_{1,i}(x_{3,i}), \theta_{1,i}(x_{4,i}),\theta_{1,i}(x_{5,i})\right)=\left(\tau_{1}^{a}, \tau_{1}^{-3a-4}, \tau_{1}^{2a}, \tau_{1}^{2}, \tau_{1}^{2}\right)$$

\noindent We may do the same process for $H_{2}=<\tau_{2}>$. It follow there exist a Fuchsian group $\Gamma_{2,i}$ with signature  $(0;2^{n+1},2^{n+1},2^{n},2^{n},2^{n})$ and an epimorphism $\theta_{2,i}:\Gamma_{1,i} \rightarrow H_{2}$. The group  $\Gamma_{2,i}$ has a presentation

$$\Gamma_{2,i}=\left<y_{1,i},\cdots,y_{5,i}: y_{1,i}^{2^{n+1}}=y_{2,i}^{2^{n+1}}=y_{3,i}^{2^{n}}=y_{4,i}^{2^{n}}=y_{5,i}^{2^{n}}=1= y_{1,i}y_{2,i}y_{3,i}y_{4,i}y_{5,i}\right>$$
where for each $j=1,2,3,4$ we have  $y_{j,i}$ is a positive minimal rotation.

\noindent Moreover there exists  an isomorphisms $f_{2,i}:\Delta/K_{2,i} \rightarrow S_{i}$, where $K_{2,i}=\ker(\theta_{2,i})$. Call $Q_{j,i}=f_{2,i}\left(\pi_{K_{2,i}}(z_{j,i})\right)$. If
$$Q_{1,i}=[1,0,1], \quad  Q_{2,i}=[-1,0,1] , \quad  Q_{3,i}=[0,0,1], \quad Q_{4,i}=[\lambda,0,1]\,  .$$
then
    \begin{enumerate}
    \item For $c$ an even number
    $$\left(\theta_{2,1}(y_{1,1}), \theta_{2,1}(y_{2,1}), \theta_{2,1}(y_{3,1}), \theta_{2,1}(y_{4,1}), \theta_{2,1}(y_{5,1})\right)=\left(\tau_{2}^{a}, \tau_{2}^{a}, \tau_{2}^{2a}, \tau_{2}^{2}, \tau_{2}^{2}\right)$$

    \item For $c$  an odd number
    $$\left(\theta_{2,2}(y_{1,2}), \theta_{2,2}(y_{2,2}), \theta_{2,2}(y_{3,2}),\theta_{2,2}(y_{4,2}),\theta_{2,2}(y_{5,2})\right)=\left(\tau_{2}^{a}, \tau_{2}^{a(2^{n}+1)}, \tau_{2}^{2a}, \tau_{2}^{2}, \tau_{2}^{2}\right)$$

    \end{enumerate}

\begin{theo}\label{example}
Assume that $2^{n}-3a-5\neq  0$ and $a > 1$ (case (4) and case (5) theorem \ref{defi}). Then the actions induced by $H_{1}$ and $H_{2}$ on $S_{i} \in \mathfrak{S}_{i}$ for each  $i=1,2$, are directly topologically, but not conformally, equivalent, except for $S_{i}$ defined by $\lambda=\pm 1 \pm \sqrt{2}$.
\end{theo}
\begin{proof}
\noindent We consider the isomorphism $\Phi :H_{j}  \rightarrow \mathbb{Z}/2^{n+1} \mathbb{Z} $ and  given by $\Phi(\tau_{j})=1$.

\noindent With this isomorphism we may associate to each generating vector a 5-tuple of elements in $\mathbb{Z}/2^{n+1} \mathbb{Z}$.
\noindent Thus we have
\begin{itemize}
    \item If $c$ is an even number, then
    $$(a,a,2a,2,2) \equiv (a,-3a-4,2a,2,2) \mod 2^{n+1}\, .$$
    \item If $c$ is an odd number, then
    $$(a,a(2^{n}+1),2a,2,2) \equiv (a,-3a-4,2a,2,2) \mod 2^{n+1}\, .$$
\end{itemize}

\noindent By theorem \ref{topocyc} we have in these cases $s=1$, then the actions are directly topologically equivalent.

\noindent Now we will prove that for $S_{i}$ defined by $\lambda \neq \pm 1 \pm \sqrt{2}$ the actions on $S_{i}$ are not conformally equivalent.

\noindent We will follow the idea of the proof given by G.Gonz\'alez-Diez and R.Hidalgo \cite{gabrub} in the case $n=2$ with $a=1$ and $c=2$.

\noindent By contradiction we suppose that are conformally equivalent; that is, there exists  $\sigma$ on $\auto(S_{i})$
such that $\sigma \tau_{1}=\tau_{2}^{j} \sigma\,$ , where $j$ is an odd number.

\noindent  Now we consider the holomorphic branched covering  associated to the action of $H_{3}=<\tau_{1}^{2}>$ on $S_{i}$, this is $\pi:  S_{i} \longrightarrow \mathbb{C}$ given by $\pi(x,y)=x$.

\noindent Hence we have the following diagram
    \[\begin{diagram}
    \node{S_{i}} \arrow{s,l}{ \pi } \arrow{e,t}{\sigma} \node{S_{i}}  \arrow{s,r}{\pi} \\
    \node{S_{i}/H_{3}=\widehat{\mathbb{C}}} \arrow{e,t,..} {T} \node{\widehat{\mathbb{C}}=S_{i}/H_{3}}
    \end{diagram}\]
where $T$  is given by $T(x):=\pi\left(\sigma(x,y)\right)$\, , for $x\in \widehat{\mathbb{C}}$, where $(x,y)\in \pi^{-1}(x)$.

\noindent It is not difficult prove that $T$ is a M\"obius transformation.

\noindent  Consider the set $B$ of the fixed points of $H_{3}$ (\eqref{ramipoints}).

\noindent Note that $\pi(B)= \{0,\infty,\pm 1, \pm \lambda, \pm \lambda^{-1}\}$ and $\sigma(B)=B$ then $T\left(\pi(B)\right)=\pi(B)\, .$ Furthermore $T(0),T(\infty) \in \{1,-1\}$. This is because $\sigma$ map the fixed points of $\tau_{1}$ on the fixed points of $\tau_{2}$.

\noindent Now  consider the coverings associated to the subgroups $H_{1}$ and $H_{2}$. Then we have the following commutative diagram
    \[\begin{diagram}
    \node{\widehat{\mathbb{C}}} \arrow{s,l}{\pi_{1}} \arrow{e,t}{T}
    \node{\widehat{\mathbb{C}}} \arrow{s,r}{\pi_{2}} \\
    \node{S_{i}/H_{1}=\widehat{\mathbb{C}}}  \arrow{e,t,..} {R}
    \node{\widehat{\mathbb{C}}=S_{i}/H_{2}}
    \end{diagram}\]
where $\pi_{1}(x)=x^{2}$, $\pi_{2}(x)=x+\frac{1}{x}$ and $R(x)=\pi_{2}T(x_{0})$, with $\pi_{1}(x_{0})=x$.

\noindent For $T_{j}:\widehat{\mathbb{C}} \longrightarrow \widehat{\mathbb{C}}$ ($j=1,2$)  defined by $T_{1}(x)=-x$ and $T_{2}(x)=\frac{1}{x}$ it has that  $T_{j}\pi = \pi \tau_{j} $ ($j=1,2$). Remark that for each $j$, $\pi_{j}$ is the covering associated to the action $T_{j}$. Using these properties we may prove $R$ is a M\"obius transformation.

\noindent Recall   $T(\pi(B))=\pi(B)=\{0,\infty,1, \lambda^{2}, \lambda^{-2}\}\, $ and $T(\{0,\infty\})=\{1,-1\}$. Then we have
\begin{equation}\label{conditionR}
        R\left(\{0,\infty,1, \lambda^{2}, \lambda^{-2}\}\right)=\{\infty , \pm 2 ,\pm \left(\lambda+\lambda^{-1}\right)\}\, ,
\end{equation}
and  $R(0),R(\infty) \in \{2,-2\}$.

\noindent In the case  $R(0)=2$\, ,  $R(\infty)=-2$ and $R(1)=\infty$, we have $R(z)=\dfrac{2z+2}{-z+1}$. In particular, $R(\lambda^{2})=\dfrac{2(\lambda^{2}+1)}{1-\lambda^{2}}$, which must be (by \eqref{conditionR} ) equal to $\pm (\lambda + \lambda^{-1})$. From this, and the fact that $\lambda^{2} \neq 1$, we deduce that $\lambda=\pm 1\pm \sqrt{2}$.
\noindent In the other case we have the same values as before for $\lambda$.

\noindent It follow that the action induced by $H_{1}$ and $H_{2}$ are not conformally equivalent for $\lambda \neq \pm 1 \pm \sqrt{2}$.

\end{proof}

\noindent Remark in the case (1), theorem \ref{defi} ($n=3$ \, ,  $a=1$ and $c=1$), it follow the  consequences of the preceding theorem. In this case the generating vectors are
 $$H_{1}: (\tau_{1}, \tau_{1}^{9}, \tau_{1}^{2}, \tau_{1}^{2}, \tau_{1}^{2}) \, ,  \quad H_{2}: (\tau_{2}, \tau_{2}^{9}, \tau_{2}^{2}, \tau_{2}^{2}, \tau_{2}^{2})$$
and the proof that the actions are not conformally equivalent is the same as that for the preceding theorem.

\bigskip
\noindent If $\lambda= \lambda_{0}=1+\sqrt{2}$, then the automorphism $\sigma$ is given by
$$\sigma(x,y)=\left(\frac{1-x}{1+x},\frac{\sqrt{2^{c}}\, \omega_{2^{n+1}}y}{(x+1)^{c}}\right)$$

\noindent  In both cases ($c$ is an even number or $c$ is an odd number) the group generated by $\tau_{1}$\, , $\tau_{2}$ and $\sigma$ acts on the Riemann surfaces defined by $\lambda_{0}$ with signature $(0;2^{n+1},2^{n+1},4)$ .

\end{document}